\documentclass[reqno,11pt]{amsart}

\usepackage{amsmath,amsfonts,amssymb,amsthm,epsfig,amscd,}
\usepackage[]{fontenc}
\usepackage{enumerate}
\usepackage[cmtip,all]{xy}
\usepackage{tabularx,multirow}

 \allowdisplaybreaks \numberwithin{equation}{section}

\newtheorem{theorem}{Theorem}[section]
\newtheorem{proposition}[theorem]{Proposition}

\newtheorem{lemma}[theorem]{Lemma}
\newtheorem{corollary}[theorem]{Corollary}
\newtheorem{claim}[theorem]{Claim}

\theoremstyle{definition}

\newtheorem{example}[theorem]{Example}

\newtheorem{remark}[theorem]{Remark}

\newcommand{\ord}{\mathrm{ord}}
\newcommand{\Aut}{\mathrm{Aut}}

\newcommand{\Sing}{\mathrm{Sing}}
\newcommand{\Span}{\mathrm{Span}}

\begin{document}

\title[]{Subcanonical points on projective curves and triply periodic minimal surfaces in the Euclidean space}
\author{Francesco Bastianelli}
\address{Dipartimento di Matematica e Fisica, Universit\`a degli Studi Roma Tre, Largo San Leonardo Murialdo 1, 00146 Roma - Italy}
\email{bastiane@mat.uniroma3.it}
\thanks{This work was partially supported by INdAM (GNSAGA); PRIN 2012 \emph{``Moduli, strutture geometriche e loro applicazioni''}; FAR 2013 (PV) \emph{``Variet\`a algebriche, calcolo algebrico, grafi orientati e topologici''}; FIRB 2012 \emph{``Spazi di moduli e applicazioni''}.}
\author{Gian Pietro Pirola}
\address{Dipartimento di Matematica, Universit\`a degli Studi di Pavia, via Ferrata 1, 27100 Pavia - Italy}
\email{gianpietro.pirola@unipv.it}

\begin{abstract}
A point $p\in C$ on a smooth complex projective curve of genus $g\geq 3$ is subcanonical if the divisor $(2g-2)p$ is canonical.
The subcanonical locus $\mathcal{G}_g\subset \mathcal{M}_{g,1}$ described by pairs $(C,p)$ as above has dimension $2g-1$ and consists of three irreducible components.
Apart from the hyperelliptic component $\mathcal{G}_g^{\textrm{hyp}}$, the other components $\mathcal{G}_g^{\textrm{odd}}$ and $\mathcal{G}_g^{\textrm{even}}$ depend on the parity of $h^0(C,(g-1)p)$, and their general points satisfy $h^0(C,(g-1)p)=1$ and $2$, respectively.
In this paper, we study the subloci $\mathcal{G}_g^{r}$ of pairs $(C,p)$ in $\mathcal{G}_g$ such that $h^0(C,(g-1)p)\geq r+1$ and $h^0\left(C,(g-1)p\right)\equiv r+1\,(\textrm{mod}\,2)$.
In particular, we provide a lower bound on their dimension, and we prove its sharpness for $r\leq 3$.
As an application, we further give an existence result for triply periodic minimal surfaces immersed in the 3-dimensional Euclidean space, completing a previous result of the second author.
\end{abstract}

\maketitle
\section{Introduction}\label{section INTRODUCTION}
The study of subcanonical points involves two very classical topics in the theory of algebraic curves: Weierstrass points and theta-characteristics.
Let $C$ be a smooth complex projective curve of genus $g$.
A point $p\in C$ is a \emph{subcanonical point} if the line bundle $\mathcal{O}_C\left((2g-2)p\right)$ is isomorphic to the canonical bundle $\omega_C$.
In particular, such a $p\in C$ is a Weierstrass point, that is $h^0\left(C,gp\right)=\dim H^0\left(C,gp\right)>1$, and the line bundle $\mathcal L=\mathcal{O}_C\left((g-1)p\right)$ is a theta-characteristic, as $\mathcal L^{\otimes 2}\cong \omega_C$.

Let $\mathcal{M}_{g,1}$ denote the moduli space of pointed curves of genus $g$, and let $\mathcal{G}_g\subset\mathcal{M}_{g,1}$ be the subcanonical locus, which is described by pairs $(C,p)$ such that $p$ is a subcanonical point of $C$.
It follows from the work of  Kontsevich and Zorich \cite{KZ} that for $g>3,$ $\mathcal{G}_g$ consists of three  irreducible components $\mathcal{G}_g^{\textrm{hyp}}$, $\mathcal{G}_g^{\textrm{odd}}$ and $\mathcal{G}_g^{\textrm{even}}$ having dimension $2g-1$.
The locus $\mathcal{G}_g^{\textrm{hyp}}$ parameterizes hyperelliptic curves with a Weierstrass point, whereas the components $\mathcal{G}_g^{\textrm{odd}}$ and $\mathcal{G}_g^{\textrm{even}}$ correspond to non-hyperelliptic curves such that $h^0\left(C,(g-1)p\right)$ is odd and even, respectively.

Every pointed hyperelliptic curve $(C,p)$ lying on the locus $\mathcal{G}_g^{\textrm{hyp}}$ satisfies $h^0\left(C,np\right)=\left\lfloor\frac{n+2}{2}\right\rfloor$ for any $0\leq n\leq 2g-2$ (see e.g. \cite[p. 274]{GH}), whereas the sequence of integers $h^0\left(C,np\right)$ for pointed curves in $\mathcal{G}_g^{\textrm{odd}}$ and $\mathcal{G}_g^{\textrm{even}}$ has been described only in the general case.
In particular, Bullock showed that the general point $[C,p]\in\mathcal{G}_g^{\textrm{odd}}$ satisfies $h^0\left(C,(g-1)p\right)= 1$, whereas the general $[C,p]\in\mathcal{G}_g^{\textrm{even}}$ is such that $h^0\left(C,(g-1)p\right)= 2$ and the linear series $\left|(g-1)p\right|$ is base-point-free (see \cite{Bu}).

Then it is interesting to investigate dimension and irreducibility of subsets of $\mathcal{G}_g^{\textrm{odd}}$ and $\mathcal{G}_g^{\textrm{even}}$ consisting of pointed curves whose associated theta-characteristic has higher number of sections.
In this paper we discuss such a problem, and we focus on the subloci of $\mathcal{M}_{g,1}$ defined as
\begin{equation*}\label{equation G_g^r}
\mathcal{G}_{g}^r:=\left\{[C,p]\in \mathcal{G}_{g}\left|
\begin{array}{l}
C\textrm{ is non-hyperelliptic},\, h^0\left(C,(g-1)p\right)\geq r+1 \\
\textrm{and } h^0\left(C,(g-1)p\right)\equiv r+1\,(\mathrm{mod}\,2)
\end{array}\right.\right\}.
\end{equation*}
In particular, we provide a general bound on the dimension of $\mathcal{G}_{g}^r$, and we prove its sharpness for $r\leq 3$.
Furthermore, as an application of the results obtained in the case $r=2$, we complete an existence result included in \cite{Pi} for triply periodic minimal Riemann surfaces immersed in the 3-dimensional Euclidean space.

\smallskip
It is clear from the definition that $\mathcal{G}_{g}^0=\mathcal{G}_g^{\textrm{odd}}$, $\mathcal{G}_{g}^1=\mathcal{G}_g^{\textrm{even}}$ and the locus $\mathcal{G}_{g}^r$ is contained in $\mathcal{G}_g^{\textrm{odd}}$ (resp. $\mathcal{G}_g^{\textrm{even}}$) when $r+1$ is odd (resp. even).
We note further that early bounds on the dimension of $\mathcal{G}_{g}^r$ may be obtained from the classical theory of Weierstrass points.
However they turn out to fail being sharp as $r\geq 1$ (see Remark \ref{remark EARLY BOUND}).

On the other hand, we interpret subcanonical points in terms of theta-characteristics, and we consider the loci $\mathcal{M}_{g}^r\subset \mathcal{M}_g$ of curves possessing a theta-characteristic with at least $r+1$ global sections and the same parity as $r+1$.
Then, using Harris' bound on the codimension of $\mathcal{M}_{g}^r$ in $\mathcal{M}_g$  (see \cite{Ha}), we achieve the following.
\begin{theorem}\label{theorem BOUND}
Let $r\geq 0$. Then either $\mathcal{G}_{g}^r$ is empty or any irreducible component $\mathcal{Z}$ of $\mathcal{G}_{g}^r$ has dimension
\begin{equation*}
\dim \mathcal{Z}\geq 2g-\frac{r(r-1)}{2}-1.
\end{equation*}
\end{theorem}
It follows from the description of $\mathcal{G}_g^{\textrm{odd}}$ and $\mathcal{G}_g^{\textrm{even}}$ that the latter bound is actually an equality both when $r=0,1$ with $g\geq 4$, and when $r=2$ as long as $\mathcal{G}_{g}^2$ is non-empty. In this direction, we prove the following.
\begin{theorem}\label{theorem r=2}
For any $g\geq 6$, the locus $\mathcal{G}_g^{2}$ is non-empty and every irreducible component $\mathcal{G}_{g}^2$ has dimension $2g-2$.
\end{theorem}
In particular, we construct a non-empty component of $\mathcal{G}_g^{2}$ whose general point $[C,p]$ is such that $h^0\left(C,(g-1)p\right)= 3$ and the linear series $\left|(g-1)p\right|$ is base-point-free (see Corollary \ref{corollary r=2 and 3}). We shall describe later how this fact is involved in the application to triply periodic minimal surface.

As far as higher values of $r$ are concerned, we prove the sharpness of the bound in Theorem \ref{theorem BOUND} for $r=3$ also. Namely,
\begin{theorem}\label{theorem r=3}
For any $g\geq 9$, the locus $\mathcal{G}_g^{3}$ is non-empty and there exists an irreducible component $\mathcal{Z}\subset\mathcal{G}_{g}^3$ of dimension $2g-4$.
\end{theorem}
Moreover, we still determine the sequence of integers $h^0\left(C,np\right)$ at the general point $[C,p]$ of the irreducible component we construct (see Corollary \ref{corollary r=2 and 3}).
We also study the locus $\mathcal{G}^3_g$ when $g=9$ and 10, showing that it is reducible and describing an irreducible component exceeding the `expected dimension' given by Theorem \ref{theorem BOUND} (see Section \ref{section EXAMPLES}).

\smallskip
Theorems \ref{theorem r=2} and \ref{theorem r=3} are analogous to the results in \cite{Fa,Te} on the sharpness of Harris' bound on the codimension of $\mathcal{M}^r_g$ in $\mathcal{M}_g$.
Furthermore, the argument we use to deduce our theorems follows the approach of Farkas, and involves also some of his technical results (see \cite[Section 2]{Fa}).
Namely, we argue by induction on $g$ and, beside the description of the base examples, we show that the existence of an irreducible component of $\mathcal{G}^r_{g-1}$ having the `expected dimension' leads to the existence of a component of $\mathcal{G}^r_{g}$ still reaching equality in Theorem \ref{theorem BOUND} (see Theorem \ref{theorem INDUCTION}).
In order to prove such a result, we combine Cornalba's theory of spin curves \cite{Cor} and Eisenbud-Harris' theory of limit linear series \cite{EH1,EH2}, and we extend subcanonical points to these settings.
In particular, this provides a criterium for smoothing `limit subcanonical points' and preserving their sequence of Weierstrass gaps (see Corollary \ref{corollary SMOOTHING WEIERSTRASS}).

We note finally that the sharpness of the bound on $\mathcal{M}^r_g$ has been recently proved in \cite{Be} for all $r\geq 2$ and $g\geq {r+2 \choose 2}$, according to  \cite[Conjecture 3.4]{Fa}.
In this line, it seems natural to conjecture the sharpness of the bound in Theorem \ref{theorem BOUND} for the very same values of $r$ and $g$.

\smallskip
The original motivation of this work
was the study of triply periodic minimal surfaces in the 3-dimensional Euclidean space $\mathbb{R}^3.$ They are very classical geometric objects, which have been studied long since (see e.g. \cite{Sc,Me}).
We say that a compact connected Riemann surface $M$ of genus $g\geq 3$ is \emph{periodic} if admits a conformal minimal immersion $\iota\colon M\longrightarrow \mathbb{T}^3$ into a flat 3-dimensional torus $\mathbb{T}^3=\mathbb{R}^3/\Lambda$.
Then a proper \emph{triply periodic} minimal surface $S$ in $\mathbb{R}^3$ is the inverse image under the universal covering $\mathbb{R}^3\longrightarrow \mathbb{T}^3$ of some immersed minimal surface $\iota(M)\subset \mathbb{T}^3$.
Following \cite{Pi}, we discuss two related problems on this topic.
On one hand, we are aimed at describing the locus $\mathcal{R}_g$ of isomorphism classes of periodic surfaces $[M]\in\mathcal{M}_g$.
On the other hand, we investigate the existence of minimal surfaces in any flat 3-dimensional real torus.

By Generalized Weierstrass Representation (see e.g. \cite{AP}), each periodic surface $M$ carries a theta-characteristic---or, equivalently a spin structure---$\mathcal{L}$ such that $h^0(M,\mathcal{L})\geq 2$ and the linear series $|\mathcal{L}|$ is base-point-free.
Thus the locus $\mathcal{R}_g$ consists of two disjoint loci $\mathcal{R}^{\textrm{even}}_g\subset \mathcal{M}^1_g$ and $\mathcal{R}^{\textrm{odd}}_g\subset \mathcal{M}^2_g$ parameterizing \emph{even} and \emph{odd} periodic surfaces, respectively.
We also denote by $\mathcal{R}^{\textrm{hyp}}_g$ the locus of hyperelliptic periodic surfaces, and we say that a proper triply periodic minimal surface in $\mathbb{R}^3$ is even, odd or hyperelliptic if the underlying periodic surface is.

Furthermore, among Riemann surfaces $M$ possessing a theta-characteristic $\mathcal{L}$ as above, periodic surfaces are characterized by a particular condition on their real periods.
In~\cite{Pi}, the latter condition has been interpreted in terms of first order deformations of the associated real period matrix.
Using the results in \cite{MR} and \cite{Na}, the existence of periodic surfaces in $\mathcal{M}_g$ follows from the existence of pointed curves $[C,p]\in \mathcal{G}_g\subset \mathcal{M}_{g,1}$ such that $h^0\left(C,(g-1)p\right)\geq 2$ and the linear series $\left|(g-1)p\right|$ is base-point-free.
Thus the description of the subcanonical loci $\mathcal{G}^{1}_g$, $\mathcal{G}^{2}_g$ and $\mathcal{G}^{\textrm{hyp}}_g$ leads to the following density result, which extends \cite[Theorem 1]{Pi} to odd periodic surfaces of genus $g\geq 7$.
\begin{theorem}\label{theorem DENSITY}
The following hold:
\begin{itemize}
  \item[(i)] If $g\geq 3$, then $\mathcal{R}^{\mathrm{even}}_g$ is dense in $\mathcal{M}^{1}_g$.
  \item[(ii)] If $g\geq 6$, then $\mathcal{R}^{\mathrm{odd}}_g$ is dense in $\mathcal{M}^{2}_g$.
  \item[(iii)] If $g\geq 3$ and $g$ is odd, then $\mathcal{R}^{\mathrm{hyp}}_g$ is dense in the hyperelliptic locus $\mathcal{M}^{\mathrm{hyp}}_g$.
\end{itemize}
\end{theorem}
Analogously, by means of the very same argument of \cite[Theorem 2]{Pi} relying on isogenies between flat $3$-dimensional tori, we achieve the following existence result.
\begin{theorem}\label{theorem MINIMAL SURFACES}
Any flat 3-dimensional torus contains countably many distinct genus $g$ immersed compact even minimal surfaces if $g\geq 3$, odd minimal surfaces if $g\geq 6$, and hyperelliptic minimal surfaces if $g\geq 3$ is odd.

The 3-dimensional Euclidean space contains countably many 6-dimensional families of genus $g$ proper triply periodic even minimal surfaces if $g\geq 3$, odd minimal surfaces if $g\geq 6$, and hyperelliptic minimal surfaces if $g\geq 3$ is odd.
\end{theorem}

The paper is organized as follows.
In Section \ref{section PRELIMINARIES} we recall some basic facts and we prove preliminary results involved in the analysis of subcanonical loci in $\mathcal{M}_{g,1}$.
Section \ref{section EXAMPLES} mainly concerns the description of the loci $\mathcal{G}^2_6$ and $\mathcal{G}^3_9$, which are the base cases to develop our inductive argument.
Section \ref{section LOCI OF SUBCANONICAL POINTS} is devoted to prove Theorems \ref{theorem BOUND}, \ref{theorem r=2}, \ref{theorem r=3} and certain additional results concerning the subcanonical locus and its boundary in $\overline{\mathcal{M}}_{g,1}$.
Finally, in Section \ref{section PERIODIC MINIMAL SURFACES} we turn to periodic minimal surfaces and, after connecting them to theta-characteristics, we retrace the argument of \cite{Pi} to prove Theorems \ref{theorem DENSITY} and \ref{theorem MINIMAL SURFACES}.

\section{Preliminaries}\label{section PRELIMINARIES}

This section concerns the preliminary notions and results necessary to discuss the loci $\mathcal{G}^r_{g}$ and to prove the related theorems we stated in the Introduction.
In particular, we shall deal with Weierstrass points, limit linear series and limits of theta-characteristics, with a view toward subcanonical points.

\subsection*{Notation}

We work throughout over the field $\mathbb{C}$ of complex numbers, unless otherwise stated.
By \emph{curve} we mean a complete connected reduced algebraic curve over $\mathbb{C}$.
We denote by $h^0(C, \mathcal{L})$ the dimension of the space $H^0(C, \mathcal{L})$ of global sections of a line bundle $\mathcal{L}$ on a curve $C$.\\

\subsection{Weierstrass points}\label{subsection WEIERSTRASS POINTS}

Let $C$ be a smooth curve of genus $g$ and let ${L=(\mathcal{L},V)}$ be a $\mathfrak{g}^r_d$ on $C$, that is a line bundle $\mathcal{L}$ of degree $d$ endowed with a $(r+1)$-dimensional subspace ${V\subset H^0(C,\mathcal{L})}$.
As in \cite{EH1}, given a point $p\in C$, the set $\displaystyle \left\{\ord_p(s)\left|s\in V\right.\right\}$ of orders of vanishing of sections of $V$ at $p$ consists of exactly $r+1$ distinct integers ${0\leq a_0^L(p)<a_1^L(p)<\dots<a_r^L(p)\leq d}$, and the sequence ${a^L(p):=\left(a_0^L(p),a_1^L(p),\dots,a_r^L(p)\right)}$ is called the \emph{vanishing sequence} of $L$ at $p$.
The \emph{ramification sequence} of $L$ at $p$ is the sequence ${\alpha^L(p):=\left(\alpha_0^L(p),\alpha_1^L(p),\dots,\alpha_r^L(p)\right)}$ such that $\alpha_i^L(p):=a_i^L(p)-i$ for any $0\leq i\leq r$, and the \emph{weight} of $p$ with respect to $L$ is defined as the integer
\begin{equation*}
w^L(p):=\sum_{i=0}^r \alpha_i^L(p).
\end{equation*}

In the case of the canonical series $K_C:=\left(\omega_C, H^0(C, \omega_C)\right)$, the sequence $\left(k_0(p),\dots,k_{g-1}(p)\right)$ such that $k_i(p):=a_i^{K_C}(p)+1$ coincides with the \emph{sequence of Weierstrass gaps} of $p$, which is the increasing sequence of integers $0\leq n\leq 2g-1$ such that $h^0\left(C, \mathcal{O}_C((n-1)p)\right)=h^0\left(C, \mathcal{O}_C(n p)\right)$.
Moreover, the set of non-gaps $\mathbb{N}\smallsetminus \left\{k_0(p),\dots,k_{g-1}(p)\right\}$ forms a semigroup with respect to addition.
We say that a point $p\in C$ is a \emph{Weierstrass point} of $C$ if its Weierstrass gaps sequence differs from $(1,2,\dots,g)$.
Hence, it follows from the definition that a point $p\in C$ is a subcanonical point if and only if it is a Weierstrass point having $k_{g-1}(p)=2g-1$, i.e. $a_{g-1}^{K_C}(p)=2g-2$ and $\alpha_{g-1}^{K_C}(p)=g-1$.

Finally, let $\mathcal{M}_{g,1}$ denote the moduli space of pointed smooth curves of genus $g$, and let ${\alpha=\left(\alpha_0,\dots,\alpha_{g-1}\right)}$ be the ramification sequence of the canonical series $K_{C}$ at a point $p\in C$, with  $[C,p]\in \mathcal{M}_{g,1}$.
Then every component of the locally closed subset $\displaystyle \mathcal{C}_{\alpha}=\left\{[Y,y]\in \mathcal{M}_{g,1}\left|\alpha^{K_Y}_i(y)=\alpha_i \textrm{ for all } 0\leq i\leq g-1\right.\right\}$ has codimension at most $\displaystyle w^{K_{C}}(p)=\sum_{i=0}^{g-1}\alpha_i$ (see e.g. \cite{EH2}).

\begin{remark}\label{remark EARLY BOUND}
The bound on the codimension of $\mathcal{C}_\alpha$ could lead to an estimate of $\dim \mathcal{G}^r_g$.
According to the semigroup condition, the ramification sequence of minimal weight for a point $[C,p]\in \mathcal{G}^r_g$ is
\begin{equation}\label{equation MINIMAL RAMIFICATION}
\alpha=(\underbrace{0,\dots,0}_{g-r-1},\underbrace{r,\dots,r}_{r},g-1).
\end{equation}
Then the existence of a pair $(C,p)$ having such a ramification sequence would give $\dim \mathcal{G}^r_g=\dim \mathcal{C}_{\alpha}\geq 3g-2-w^{K_C}(p)=2g-1-r^2$.
In \cite{Bu}, it has been proved that the ramification sequences of a general point of $\mathcal{G}^0_g$ and $\mathcal{G}^1_g$ does satisfy (\ref{equation MINIMAL RAMIFICATION}).
For $r=0$, we then get  $\dim \mathcal{G}^0_g\geq  2g-1$, which is actually an equality.
However, in the light of Theorem \ref{equation MINIMAL RAMIFICATION}, the bound above turns out to be no longer sharp as $r\geq 1$ (see also \cite[p. 497]{EH2} and \cite[Remark 2]{Bu}).
\end{remark}

\subsection{Limit linear series}\label{subsection LIMIT LINEAR SERIES}

Following \cite{EH1}, we briefly review the notion of limit linear series on curves of compact type.
In particular, we focus on nodal curves of the form $X=C_1\cup_q C_2$ given by two smooth curves $C_1$ and $C_2$ of genus $g-j$ and $j$ respectively, meeting at a single ordinary node $q$.

A \emph{crude limit} $\mathfrak{g}^r_d$ on $X$ is a collection $L={\left\{\left(\mathcal{L}_{C_1}, V_{C_1}\right),\left(\mathcal{L}_{C_2}, V_{C_2}\right)\right\}}$, where $L_{C_1}:=\left(\mathcal{L}_{C_1}, V_{C_1}\right)$ and $L_{C_2}:=\left(\mathcal{L}_{C_2}, V_{C_2}\right)$ are $\mathfrak{g}^r_d$ on $C_1$ and $C_2$ respectively, such that the vanishing sequences $a^{L_{C_1}}(q)$ and $a^{L_{C_2}}(q)$ satisfy the \emph{compatibility conditions}
\begin{equation}\label{equation COMPATIBILITY CONDITIONS}
a_i^{L_{C_1}}(q) + a_{r-i}^{L_{C_2}}(q)\geq d\quad\textrm{for any } 0\leq i\leq r.
\end{equation}
In particular, given a point $p\in C_2-\{q\}$ (resp. $p\in C_1-\{q\}$), we define the vanishing and the ramification sequences of $L$ at $p$ as ${a^{L}(p):=a^{L_{C_2}}(p)}$ and ${\alpha^{L}(p):=\alpha^{L_{C_2}}(p)}$ (resp. ${a^{L}(p):=a^{L_{C_1}}(p)}$ and ${\alpha^{L}(p):=\alpha^{L_{C_1}}(p)}$).
Moreover, we say that $L$ is a \emph{refined limit} $\mathfrak{g}^r_d$, or simply a \emph{limit} $\mathfrak{g}^r_d$, if all the inequalities in (\ref{equation COMPATIBILITY CONDITIONS}) are equalities.

It is worth recalling that in \cite[Section 3]{EH1} are constructed the scheme of limit $\mathfrak{g}^r_d$ on a curve of compact type $X$, and more generally schemes parameterizing limit $\mathfrak{g}^r_d$ varying on families of pointed curves of compact type.
\begin{remark}\label{remark UPPER SEMI-CONTINUITY}
Let $(X,p)$ be a pointed singular curve of compact type and let $L$ be a limit $\mathfrak{g}^r_d$ as above.
Suppose that the triple $(X,p,L)$ is smoothable, that is there exists a family of smooth pointed curves endowed with a $\mathfrak{g}^r_d$ having $(X,p,L)$ as central fibre.
Denoting by $B$ the base space of the family and by $(X_b,p_b,L_b)$ the fibre over $b\in B$, we have that the vanishing and the ramification sequences of $L_b$ at $p_b\in X_b$ satisfy upper semi-continuity, that is $a_i^{L_b}(p_b)\leq a_i^L(p)$ and $\alpha_i^{L_b}(p_b)\leq \alpha_i^L(p)$ for any $0\leq i\leq r$ and for any $b\in B$.
\end{remark}

We recall that $\mathcal{M}^r_{g}$ is the sublocus of $\mathcal{M}_g$ which parameterizes curves $C$ admitting a theta-characteristic $L$ such that $h^0\left(C,L\right)\geq r+1$ and $h^0\left(C,L\right)\equiv r+1\,(\mathrm{mod}\,2)$.
The following lemma describes `limit theta-characteristics' in terms of limit $\mathfrak{g}^r_{g-1}$ on curves lying in the boundary of $\overline{\mathcal{M}}^r_{g}$ (see \cite[Lemma 2.2]{Fa}).
\begin{lemma}\label{lemma FARKAS LIMIT g^r_(g-1)}
Let $[C_1\cup_q C_2]\in \overline{\mathcal{M}}^r_{g}$.
Then the curve $C_1\cup_q C_2$ possesses a limit $\mathfrak{g}^r_{g-1}$ ${\displaystyle L=\left\{\left(\mathcal{L}_{C_1},V_{C_1}\right),\left(\mathcal{L}_{C_2},V_{C_2}\right)\right\}}$ such that
\begin{equation}\label{equation GENERAL LOCALLY CLOSED}
\mathcal{L}_{C_1}^2\cong \omega_{C_1}(2j\,q)\quad \textrm{and} \quad \mathcal{L}_{C_2}^2\cong \omega_{C_2}\left(2(g-j)q\right),
\end{equation}
where $g-j$ and $j$ are the genus of ${C_1}$ and ${C_2}$, respectively.
\end{lemma}

\begin{remark}\label{remark LIMIT SUBCANONICAL POINTS}
Given a smooth pointed curve $[Y,y]\in\mathcal{M}_{g,1}$, the point $y\in Y$ is subcanonical if and only if there exists some complete $\mathfrak{g}^r_{g-1}$ $L$ on $Y$ such that $L^2\cong \omega_Y$ and $a_r^L(y)=g-1$.
In the previous lemma, conditions (\ref{equation GENERAL LOCALLY CLOSED}) extend the notion of theta-characteristic to limit $\mathfrak{g}^{r}_{g-1}$ on nodal curves (see also \cite[p. 363]{EH1}).
Analogously, upper semi-continuity gives that for any $[C_1\cup_q C_2,p]\in\overline{\mathcal{G}}^r_{g}$, the curve $C_1\cup_q C_2$ possesses a limit $\mathfrak{g}^r_{g-1}$ $L$ satisfying (\ref{equation GENERAL LOCALLY CLOSED}) and the extra condition $a_r^L(p)=g-1$.
\end{remark}

Let $\Delta_{1,1}\subset \overline{\mathcal{M}}_{g,1}$ denote the irreducible boundary component whose general point $[C\cup_q E,p]$ is such that $C$ has genus $g-1$, $E$ is an elliptic curve and $p\in E$.
In the spirit of \cite[Lemma 2.2]{Bu}, we prove the following lemma, which describes pointed curves on $\Delta_{1,1}$ having a `limit subcanonical point' whose associated `limit theta-characteristic' has $r+1$ sections.
\begin{lemma}\label{lemma LIMIT g^r_(g-1)}
Fix $g\geq 3$ and $r\geq 1$, and let $[C\cup_q E,p]\in \Delta_{1,1}\subset \overline{\mathcal{M}}_{g,1}$.
\begin{itemize}
  \item[(i)] Let $\displaystyle L=\left\{\left(\mathcal{L}_C,V_C\right),\left(\mathcal{L}_E,V_E\right)\right\}$ be a limit $\mathfrak{g}^r_{g-1}$ on $C\cup_q E$ satisfying
      \begin{equation}\label{equation LOCALLY CLOSED}
      \mathcal{L}_{C}^2\cong \omega_{C}(2q)\quad \textrm{and} \quad \mathcal{L}_{E}^2\cong \omega_{E}\left(2(g-1)q\right),
      \end{equation}
      and such that $a_r^L(p)= g-1$.
      Then $q\in C$ is a subcanonical point with $h^0(C,(g-2)q)\geq r$, and $p-q\in E$ is a torsion point of order $2g-2$.
  \item[(ii)] Conversely, if $q\in C$ is a subcanonical point with $h^0(C,(g-2)q)\geq r+1$, and $p-q\in E$ is a torsion point of order exactly $2g-2$,
  then the curve $C\cup_q E$ admits a limit $\mathfrak{g}^r_{g-1}$ ${\displaystyle L=\left\{\left(\mathcal{L}_C,V_C\right),\left(\mathcal{L}_E,V_E\right)\right\}}$ satisfying conditions (\ref{equation LOCALLY CLOSED}) and such that $a_r^L(p)= g-1$.
\end{itemize}
In particular, $\mathcal{L}_C=\mathcal{O}_C((g-1)q)$ and $\mathcal{L}_E=\mathcal{O}_E((g-1)p)$ in both $\mathrm{(i)}$ and $\mathrm{(ii)}$.
\end{lemma}
\begin{proof}
In order to prove (i), we note that $L_E=\left(\mathcal{L}_E,V_E\right)$ is a $\mathfrak{g}^r_{g-1}$ on $E$ having vanishing order $a_r^{L_E}(p)=a_r^{L}(p)=g-1$ at $p\in E$.
Thus there exists $\sigma_E\in V_E$ such that $\ord_p(\sigma_E)=g-1$, and hence $\mathcal{O}_E\left((g-1)p\right)\cong \mathcal{L}_E$.
Thanks to conditions (\ref{equation LOCALLY CLOSED}) we have $\mathcal{O}_E\left((2g-2)p\right)\cong \mathcal{L}_E^2\cong \mathcal{O}_E\left((2g-2)q\right)$, that is $p-q\in E$ is a torsion point of order $2g-2$.\\
Since $\ord_p(\sigma_E)=g-1$, we have $\ord_q(\sigma_E)=0$ and hence $a_0^{L_E}(q)=0$.
Then $a_r^{L_C}(q)=g-1$ by compatibility conditions (\ref{equation COMPATIBILITY CONDITIONS}).
In particular, $\mathcal{O}_C\left((g-1)q\right)\cong \mathcal{L}_C$ and $h^0\left(C,(g-1)q\right)\geq \dim V_{C}=r+1$.
Therefore $h^0\left(C,(g-2)q\right)\geq r$ and $\mathcal{O}_{C}\left((2g-2)q\right)\cong \mathcal{L}^{2}_{C}\cong \omega_{C}(2q)$ by (\ref{equation LOCALLY CLOSED}).
Thus $(2g-4)q$ is a canonical divisor of the curve $C$ of genus $g-1$, i.e. $q\in C$ is a subcanonical point and assertion (i) follows.

So we turn to prove (ii).
Let  $\mathcal{L}_C:=\mathcal{O}_C((g-1)q)$ and $\mathcal{L}_E:=\mathcal{O}_E((g-1)p)$.
By assumption, $(2g-2)p$ and $(2g-2)q$ are linearly equivalent divisors on $E$, and $(2g-4)q$ is a canonical divisor on $C$.
Thus $\mathcal{L}_C$ and $\mathcal{L}_E$ satisfy conditions (\ref{equation LOCALLY CLOSED}).
Moreover, the subcanonical point $q\in C$ is such that $h^0(C,(g-2)q)\geq r+1$.
Then there exists ${\sigma\in H^0(C, \mathcal{O}_C((g-2)q))}$ satisfying $\ord_q(\sigma)=g-2$, and we can choose a $(r+1)$-dimensional subspace ${V\subset H^0(C, \mathcal{O}_C((g-2)q))}$ passing through $\sigma$.
We define ${V_C\subset H^0(C, \mathcal{O}_C((g-1)q))}$ as the image of $V$ under the natural map ${H^0(C, \mathcal{O}_C((g-2)q))\hookrightarrow H^0(C, \mathcal{O}_C((g-1)q))}$.
So the linear series $L_C:=\left(\mathcal{L}_C,V_C\right)$ has a base point at $q$, and its vanishing sequence $a^{L_C}(q)=(c_0,\dots,c_r)$ at $q$ is such that ${c_0\geq 1}$ and ${c_r=g-1}$.\\
We consider the sequences $(b_0,\dots,b_r)$ and $(a_0,\dots,a_r)$ defined as $b_i:=g-1-c_{r-i}$ for any $0\leq i\leq r$, $a_r:=g-1$ and $a_{r-i}:= g-2-b_i$ for any $1\leq i\leq r$.
Then \cite[Proposition 5.2]{EH2} assures that if
\begin{equation}\label{equation EH2}
\begin{array}{lcl}
a_{r-i}+ b_i= g-1 & \Rightarrow & a_{r-i}p+ b_iq \sim (g-1)p\\
a_{r-i}p+ (b_i+1)q \sim (g-1)p & \Rightarrow & b_{i}+1= b_{i+1}
\end{array}
\end{equation}
hold for each $i$, there exists a linear series $L_E:=\left(\mathcal{L}_E,V_E\right)$ such that $a^{L_E}(p)=(a_0,\dots,a_r)$ and $a^{L_E}(q)=(b_0,\dots,b_r)$.
Since $p-q$ is a torsion point on $E$ having order exactly $2g-2$, conditions (\ref{equation EH2}) are satisfied.
In particular, $a_r^{L}(p)=a_r=g-1$.
Moreover, $b_i+c_{r-i}=g-1$ for any  $0\leq i\leq r$ by definition, therefore $L$ is a refined limit linear series.
\end{proof}

\subsection{Theta-characteristics and spin curves}\label{subsection THETA-CHARACTERISTICS}

We summarize the  approach of  \cite{Cor} to compactify the moduli space $\mathcal{S}_g$ parameterizing isomorphism classes $[C,\mathcal{L}]$, where $[C]\in \mathcal{M}_g$ and $\mathcal{L}$ is a theta-characteristic on $C.$
To this aim, we consider a semistable curve $X$ having at most ordinary nodes as singularities and arithmetic genus $p_a(X)=g\geq 2$.
An \emph{exceptional component} is a smooth rational curve $R\subset X$ meeting the rest of the curve $X$ at two points.
We say that $X$ is a \emph{decent curve} if every two exceptional components are disjoint.
A \emph{spin curve} is a triple $(X,\xi,\alpha)$ consisting of a decent curve $X$, a line bundle $\xi$ of degree $g-1$ on $X$ such that $\xi_{|R}=\mathcal{O}_R(1)$ for any exceptional component $R\subset X$, and a homomorphism $\alpha\colon \xi^2\longrightarrow \omega_X$ being non-null on every non-exceptional component of $X$.
In particular, when $X$ is a curve of compact type, the restriction $\xi_{|C}$ gives a theta-characteristic on any smooth non-exceptional component $C\subset X$.
Moreover, any decent curve is obtained by blowing-up some nodes of a stable curve $[\widetilde{X}]\in \overline{\mathcal{M}}_g$.
Thus we define the moduli space $\overline{\mathcal{S}}_g$ parameterizing isomorphism classes $[X,\xi,\alpha]$ of spin curves, and there exists a finite map
$$
{\pi\colon \overline{\mathcal{S}}_g\longrightarrow\overline{\mathcal{M}}_g}
$$
sending $[X,\xi,\alpha]$ to the class of the stable model $[\widetilde{X}]$.
\begin{remark}
Consider a stable curve $\widetilde{X}=C_1\cup_q C_2$ whose components are two smooth curves of genus $g-j$ and $j$ meeting at a node $q$.
Then $\widetilde{X}$ does not admit any structure of spin curve, and the fibre $\pi^{-1}[\widetilde{X}]$ consists of classes $[X,\xi,\alpha]\in \overline{\mathcal{S}}_g$, where $X=C_1\cup_{q'}R\cup_{q''}C_2$ is the blow-up of $\widetilde{X}$~at~$q$ (see \cite[Example 3.1]{Cor}).
\end{remark}
A \emph{family of spin curves} is a triple $(\mathcal{X}\stackrel{\phi}{\longrightarrow} T, \xi, \alpha)$ consisting of a family of decent curves $\mathcal{X}\longrightarrow T$, a line bundle $\xi$ on $\mathcal{X}$, and a homomorphism $\alpha\colon \xi^2\longrightarrow \omega_\phi$ such that each triple $(X_t,\xi_t,\alpha_t)$ is a spin curve, where $X_t:=\phi^{-1}(t)$, $\xi_t:=\xi_{|X_t}$ and $\alpha_t:=\alpha_{|\xi_t^2}$.
The following theorem extends to spin curves two classical results on theta-characteristics due to Mumford and Harris, respectively (see \cite{Ha,Mu} and \cite[Theorem 3.3]{Col}).
\begin{theorem}\label{theorem COLOMBO}
Let $\displaystyle (\mathcal{X}\stackrel{\phi}{\longrightarrow} T, \xi, \alpha)$ be a family of spin curves.
Then
\begin{itemize}
  \item[(i)] the function $T\longrightarrow \mathbb{Z}/2\mathbb{Z}$ defined as $t\longmapsto \left[h^0(C_t,\xi_t)\right]_{(\mathrm{mod}\,2)}$ is locally constant;
  \item[(ii)] the locus ${T^r:=\left\{t\in T\left| h^0\left(X_t,\xi_t\right)\geq r+1 \textrm{ and } h^0\left(X_t,\xi_t\right)\equiv r+1 \,(\mathrm{mod}\,2) \right.\right\}}$ is either empty or it has codimension at most ${r+1 \choose 2}$.
\end{itemize}
\end{theorem}

In particular, the locus
\begin{equation*}
\mathcal{S}^r_g:=\left\{[C,\mathcal{L}]\in \mathcal{S}_g\left| h^0\left(C,\mathcal{L}\right)\geq r+1 \textrm{ and } h^0\left(C,\mathcal{L}\right)\equiv r+1\,(\mathrm{mod}\,2) \right.\right\}
\end{equation*}
has codimension at most ${r+1 \choose 2}$ in $\mathcal{S}_g$, according to the identity $\mathcal{M}^r_g=\pi\left(\mathcal{S}^r_g\right)$.

Since we shall mainly deal with pointed curves, let $\mathcal{S}_{g,1}$ denote the moduli space of pointed smooth spin curves, and let $\mathcal{S}^r_{g,1}$ be the pre-image of $\mathcal{S}^r_g$ under the forgetful morphism ${\mathcal{S}_{g,1}\longrightarrow\mathcal{S}_g}$.
We recall that a theta-characteristic $\mathcal{L}$ on a curve $C$ is said to be \emph{odd} (resp. \emph{even}) if $h^0\left(C,\mathcal{L}\right)$ is.
Moreover, the locus $\Delta_{1,1}\subset \overline{\mathcal{M}}_{g,1}$ denotes the boundary divisor parameterizing nodal curves having an elliptic tail with a marked point.
Then we define the irreducible boundary components $A_{1,1}^+$ and $A_{1,1}^-$ of $\overline{\mathcal{S}}_g$ as
\begin{equation*}
A_{1,1}^+:=\overline{\left\{[X,\xi,\alpha,p]\in \overline{\mathcal{S}}_g\left| \begin{array}{l}
\pi\left([X,\xi,\alpha,p]\right)=[C\cup_q E,p]\in \Delta_{1,1},\\
\xi_{|C} \textrm{ and } \xi_{|E} \textrm{ are even}
\end{array}\right.\right\}}
\end{equation*}
and
\begin{equation*}
A_{1,1}^-:=\overline{\left\{[X,\xi,\alpha,p]\in \overline{\mathcal{S}}_g\left| \begin{array}{l}
\pi\left([X,\xi,\alpha,p]\right)=[C\cup_q E,p]\in \Delta_{1,1},\\
\xi_{|C} \textrm{ is odd and } \xi_{|E} \textrm{ is even}
\end{array}\right.\right\}}.
\end{equation*}
Hence the partial compactification $\mathcal{S}^*_g:=\mathcal{S}_g\cup A_{1,1}^+\cup A_{1,1}^-$ maps finitely over $\mathcal{M}_g\cup \Delta_{1,1}$ through $\pi$.
Finally, we have the following result (see \cite[Proposition 2.4]{Fa}).
\begin{lemma}\label{lemma FARKAS SMOOTHING THETA}
Let $[C\cup_q E,p]\in \Delta_{1,1}$ such that $[C,q]\in \mathcal{G}^r_{g-1}$ and $p-q\in E$ is a torsion point of order exactly $2g-2$.
Let $[X,\xi,\alpha]\in \overline{\mathcal{S}}_{g}$ be the spin curve such that $X=C\cup_{q'}R\cup_{q''}E$ is the blow up of $C\cup_q E$ at q, $\xi_{|C}=\mathcal{O}_C\left((2g-4)q'\right)$ and $\xi_{|E}=\mathcal{O}_E\left((g-1)(p-q'')\right)$.
If $[C,\xi_{|C}]\in \mathcal{S}_{g-1}$ lies on an irreducible component of $\mathcal{S}_{g-1}^r$ having codimension ${r+1 \choose 2}$ in $\mathcal{S}_{g-1}$, then $[X,\xi,\alpha]\in \overline{\mathcal{S}}_{g}^r$ and $[C\cup_q E]\in \overline{\mathcal{M}}_{g}^r$.
\end{lemma}

\begin{remark}
Lemma \ref{lemma FARKAS SMOOTHING THETA} is actually included in the proof of \cite[Proposition 2.4]{Fa}, under slightly different assumptions.
In particular, the point $q$ is assumed to be general on both $C$ and $E$, and $\xi_{|E}=\mathcal{O}_E(t-s)$ for some torsion point of order two $t-s\in E$.
However, these differences do not affect the proof of Lemma \ref{lemma FARKAS SMOOTHING THETA}.
\end{remark}

\section{Examples}\label{section EXAMPLES}

This section is mainly devoted to describe the loci $\mathcal{G}^2_6$ and $\mathcal{G}^3_9$.
On one hand, they provide the first step of the inductive argument we shall perform in the next section.
On the other hand, we shall see that $\mathcal{G}^3_9$ is reducible, and we shall describe a component satisfying strict inequality in Theorem \ref{theorem BOUND}.
Moreover, we shall briefly discuss the locus $\mathcal{G}^3_{10}$, which possesses two distinct irreducible component having the `expected dimension'.

\begin{example}[$r=2$ and $g=6$]\label{example G_6^2}
It follows from the detailed analysis in \cite[Section 4.5]{Bu} that $\mathcal{G}^2_6\subset \mathcal{M}_{6,1}$ is an irreducible non-empty locus of dimension $10$, and for any point $[C,p]\in \mathcal{G}^2_6$, the canonical ramification sequence at $p$ is $\alpha^{K_C}(p)=(0,0,0,2,2,5)$.
In particular, both the dimension of $\mathcal{G}^2_6$ and the weight of $\alpha^{K_C}(p)$ reach the smallest possible values in view of Theorem \ref{theorem BOUND} and Remark \ref{remark EARLY BOUND}.

We note further that each point $[C,p]\in\mathcal{G}^2_6$ is given by a smooth curve $C\in \mathbb{P}^2$ of degree $d=5$ having a 5-fold inflection point at $p\in C$, that is there exists a line $\ell\in \mathbb{P}^2$ such that  $(C\cdot \ell)_p=5$.
Conversely, given such a quintic plane curve $C$, we have that $\omega_C\cong \mathcal{O}_C(2)$ and $|\mathcal{O}_C(1)|$ is the unique $\mathfrak{g}^2_5$ on $C$, so that any 5-fold inflection point $p\in C$ is subcanonical with $h^0(C,5p)=3$.
\end{example}

In the following examples we turn to describe $\mathcal{G}^3_9$.
We construct two distinct components $\mathcal{Z}_1,\mathcal{Z}_2\subset \mathcal{G}^3_9$ of dimension $\dim \mathcal{Z}_1=14$ and $\dim \mathcal{Z}_2=15$, whose general points $[C,p]$ satisfy ${\alpha^{K_C}(p)=(0,\dots,0,3,3,3,8)}$ and $\alpha^{K_C}(p)=(0,0,0,1,1,3,3,5,8)$, respectively.
In particular, the component $\mathcal{Z}_1$ satisfies equality in Theorem \ref{theorem BOUND}, and the canonical sequence at its general point has minimal weight.
On the other hand, the component $\mathcal{Z}_2$ exceeds the `expected dimension'.

\begin{example}[$r=3$ and $g=9$]\label{example G_9^3 minimal}
Let $\mathcal{Z}_1\subset \mathcal{M}_{9,1}$ be the locus whose general point $[C,p]$ is such that $C=S\cap T\subset \mathbb{P}^3$ is the complete intersection of a smooth quadric $S$ and a smooth quartic $T$, where $p\in C$ is an inflection point of order $8$, i.e. there exists a plane $H\subset \mathbb{P}^3$ satisfying $(C\cdot H)_p=8$.

By adjunction formula, any such a curve $C$ has genus 9 and $\omega_C\cong \mathcal{O}_C(2)$, so that $p\in C$ is a subcanonical point.
It follows from the description of \cite[p. 199-200]{ACGH} that the sequence of Weierstrass gaps of $p$ is either $(1,\dots,5,9,10,11,17)$ or $(1,2,3,5,6,9,10,13,17)$.
If it were the latter, then $|4p|$ would be the $\mathfrak{g}^1_4$ obtained by projecting $C$ from its tangent line $\ell$ at $p$.
Therefore $H\cap S=2\ell$ and $S$ would be a quadric cone, contradicting the assumption of smoothness.
Thus the sequence of Weierstrass gaps at $p$ is the former and the corresponding canonical ramification sequence is $\alpha^{K_C}(p)=(0,\dots,0,3,3,3,8)$.
In particular, $h^0(C, 8p)=4$ and hence $[C,p]\in \mathcal{G}^3_9$.

The existence of smooth curves as above can be proved as follows.
We fix in $\mathbb{P}^3$ a smooth quadric $S$ and a plane $H$ such that $D:=S\cap H$ is a smooth curve.
Then $S\cong \mathbb{P}^1\times \mathbb{P}^1$ and $D$ has bidegree $(1,1)$ on $S$.
In particular, $D$ is a smooth plane conic, so that $D\cong \mathbb{P}^1$.
Thus $\mathcal{O}_S(4)\cong \mathcal{O}_{\mathbb{P}^1\times \mathbb{P}^1}(4,4)$ and $\mathcal{O}_D(4)\cong \mathcal{O}_{\mathbb{P}^1}(8)$.
Moreover, the restriction maps $H^0\left(\mathbb{P}^3, \mathcal{O}_{\mathbb{P}^3}(4)\right)\longrightarrow H^0\left(S, \mathcal{O}_{S}(4)\right)$ and
$H^0\left(\mathbb{P}^1\times \mathbb{P}^1, \mathcal{O}_{\mathbb{P}^1\times \mathbb{P}^1}(4,4)\right)\longrightarrow H^0\left(\mathbb{P}^1, \mathcal{O}_{\mathbb{P}^1}(8)\right)$ turn out to be surjective.
Hence all the effective divisors of degree $8$ on $D$ are cut out by curves of bidegree $(4,4)$ on $S$.
Therefore, fixing a point $p\in D$, there exists some curve $C\subset S$ of bidegree $(4,4)$ meeting $D$ at $p$ with multiplicity $8$, so that $(C\cdot H)_p=8$ as well.
Furthermore, the surjections above show that $C$ must be the complete intersection of $S$ with some quartic surface $T\subset \mathbb{P}^3$.
Finally, let ${C\subset \mathbb{P}^3}$ be the complete intersection of $S:=\{xz-y^2+t^2=0\}$ and $T:=\{x^4+x^3z-x^2y^2+x^2yz-xy^3+x^2z^2+xy^2z-2y^4+xyz^2-y^3z+xz^3-y^2z^2+tz^3+t^4=0\}$.
By direct computation, it is possible to check that $C$, $S$ and $T$ are smooth.
Moreover, the plane $H:=\{t=0\}$ and the point $p:=(0:0:1:0)$ are such that $(C\cdot H)_p=8$.
Hence the linear subsystem $|V|\subset |\mathcal{O}_S(4)|$ of curves $C\subset S$ satisfying the linear condition $(C\cdot H)_p\geq 8$ is non-empty, and its general element is smooth by Bertini's Theorem.
In particular, the locus $\mathcal{Z}_1$ is non-empty as well.

To conclude, we point out that $(C\cdot H)_p\geq 8$ imposes exactly 8 independent linear conditions on $|\mathcal{O}_S(4)|$.
Thus the dimension of $\mathcal{Z}_1$ is
\begin{eqnarray*}
\dim \mathcal{Z}_1 & = & \underbrace{{2+3 \choose 3}-1}_{\textrm{choice of }S}+\underbrace{{4+3 \choose 3}-{2+3 \choose 3}-1}_{\textrm{choice of }T_{|S} }+\underbrace{1}_{\textrm{choice of }p}+\underbrace{3}_{\textrm{choice of }H} \\
& & -\underbrace{8}_{(C\cdot H)_p\geq 8}-\underbrace{15}_{\mathrm{PGL}(4)}= 14.
\end{eqnarray*}
\end{example}

\begin{example}[$r=3$ and $g=9$]\label{example G_9^3 maximal}
We define $\mathcal{Z}_2\subset \mathcal{M}_{9,1}$ as the locus of pointed curves $[C,p]$ such that $C=S\cap T\subset \mathbb{P}^3$ is the complete intersection of a quadric cone $S$ and a smooth quartic $T$, and there exists a line of the ruling $\ell\subset S$ satisfying $(C\cdot \ell)_p=4$.

By arguing as above, it is possible to check that $\mathcal{Z}_2$ is a non-empty component of $\mathcal{G}^3_9$ having $\dim \mathcal{Z}_2=15$ and $\alpha^{K_C}(p)=(0,0,0,1,1,3,3,5,8)$ at its general point.
Furthermore, the component $\mathcal{Z}_1$ is not contained in $\mathcal{Z}_2$ because of upper semi-continuity of ramification sequences.
\end{example}

To conclude, we briefly describe also the locus $\mathcal{G}^3_{10}$, which is reducible and contains two irreducible components satisfying equality in Theorem \ref{theorem BOUND}.
\begin{example}[$r=3$ and $g=10$]\label{example G_10^3}
Let $C= \Gamma_1\cap \Gamma_2\subset \mathbb{P}^3$ be a smooth curve obtained as complete intersection of two smooth cubic surfaces $\Gamma_1$ and $\Gamma_2$.
Assume that $p\in C$ is an inflection point of order 9, and let $H\subset \mathbb{P}^3$ be the plane such that $(C\cdot H)_p=9$.
By adjunction formula, the curve $C$ has genus 10 and $\omega_C\cong \mathcal{O}_C(2)$, so that $p\in C$ is a subcanonical point and $[C,p]\in \mathcal{G}^3_{10}$.
Then the configurations of the cubic plane sections $C_1:=\Gamma_1\cap H$ and $C_2:=\Gamma_2\cap H$ turn out to distinguish two non-empty components $\mathcal{Z}_1,\mathcal{Z}_2\subset\mathcal{G}^3_{10}$ of dimension $\dim \mathcal{Z}_1=\dim \mathcal{Z}_2= 2g-4=16$, as follows:
\begin{itemize}
  \item[(i)] the general point $[C,p]\in \mathcal{Z}_1$ is such that the plane cubic $C_1$ possesses a 3-inflection point at $p$, and the canonical ramification sequence is $\alpha^{K_C}(p)=(0,\dots,0,1,3,3,4,9)$;
  \item[(ii)] the general point $[C,p]\in \mathcal{Z}_2$ is such that the point $p-O$ on the elliptic curve $C_1$ is a torsion point of order 9 for some 3-inflection point $O\in C_1$, and the canonical ramification sequence is ${\alpha^{K_C}(p)=(0,\dots,0,3,3,3,9)}$.
\end{itemize}
\end{example}

\begin{remark}
We point out that, after a detailed analysis, it should be possible to argue as in Example \ref{example G_9^3 minimal} and to construct components of $\mathcal{G}^4_{13}$ and $\mathcal{G}^5_{17}$ having the `expected dimensions' $2g-7$ and $2g-11$, respectively.
As it shall be clear in the next section, these constructions would work as base cases of our inductive argument, so that Theorem \ref{theorem r=3} would be extended to both $\mathcal{G}^4_{g}$ and $\mathcal{G}^5_{g}$.\\
When $r=4$, the component $\mathcal{Z}\subset \mathcal{G}^4_{13}$ is described by complete intersections $C=S_1\cap S_2\cap T\subset \mathbb{P}^4$ of two smooth quadric hypersurfaces $S_1, S_2$ and a cubic hypersurface $T$.
When $r=5$, the component $\mathcal{Z}\subset \mathcal{G}^4_{13}$ consists instead of complete intersections $C=S_1\cap \dots\cap S_4\subset \mathbb{P}^5$ of four smooth quadrics.
\end{remark}

\section{Loci of subcanonical points}\label{section LOCI OF SUBCANONICAL POINTS}

This section concerns our main results on the subcanonical loci $\mathcal{G}_{g}^r$.
In particular, after proving Theorem \ref{theorem BOUND}, we shall present a smoothing theorem, which is the core of the inductive argument to achieve Theorems \ref{theorem r=2} and \ref{theorem r=3}.
Finally, we shall deduce some further related results.

\begin{proof}[Proof of Theorem \ref{theorem BOUND}]
Suppose that $\mathcal{G}_{g}^r$ is non-empty and let ${[C,p]\in \mathcal{G}_{g}^r}$, so that ${p\in C}$ is a subcanonical point and ${N:=\mathcal{O}_C\left((g-1)p\right)}$ is a theta-characteristic.
We want to prove that any irreducible component $\mathcal{Z}\subset \mathcal{G}_{g}^r$ passing through $[C,p]$ has dimension ${\dim \mathcal{Z} \geq  2g-\frac{r(r-1)}{2}-1}$.
To this aim, let ${\left(\mathcal{C}\stackrel{\phi}{\longrightarrow} U, \mathcal{N},  U\stackrel{\rho}{\longrightarrow} \mathcal{C}\right)}$ be a versal deformation of the pointed smooth curve $(C,N,p)$ in $\mathcal{S}_{g,1}$.
Therefore ${U/\Aut(C,N,p)\hookrightarrow \mathcal{S}_{g,1}}$, the family ${\phi\colon \mathcal{C}\longrightarrow U}$ consists of smooth curves $C_t:=\phi^{-1}(t)$ of genus $g$, the line bundle $\mathcal{N}$ on $\mathcal{C}$ is such that ${N_t:=\mathcal{N}_{|C_t}}$ is a theta characteristic on $C_t$, the map ${\rho\colon U\longrightarrow \mathcal{C}}$ is a section of $\phi$ with ${p_t:=\rho(t)\in C_t}$, and $(C_{0},N_{0},p_{0})=(C,N,p)$ for some point $0\in U$.

Then we consider the $(g-1)$-fold relative symmetric product ${\mathcal{C}^{(g-1)}\longrightarrow U}$ of the family $\mathcal{C}$, so that the fibre over each $t$ is the $(g-1)$-fold symmetric product $C_t^{(g-1)}$ of the curve $C_t$.
Setting ${U^r:=\left\{t\in U\left|[C_t,N_t,p_t]\in \mathcal{S}_{g,1}^r\right.\right\}}$, we define two subvarieties of $\mathcal{C}^{(g-1)}$ as follows.
The first one is
\begin{equation*}
{\mathcal{Y}:=\left\{\left. p_1+\dots+p_{g-1}\in C_t^{(g-1)}\right|\,t\in U^r \textrm{ and } \mathcal{O}_{C_t}\left(p_1+\dots+p_{g-1}\right)\cong N_t\right\}},
\end{equation*}
which consists of effective divisors $D_t\in C_t^{(g-1)}$ that induce the theta-characteristic $N_t$ on $C_t$ having at least $r+1$ global sections and the same parity as $r+1$.
The second subvariety is ${\mathcal{P}:=\left\{\left.p_t+\dots+p_t\in C_t^{(g-1)}\right|t\in U\right\}}$, which is the image of the map $\rho^{(g-1)}\colon U\longrightarrow \mathcal{C}^{(g-1)}$ induced by $\rho$.
Therefore, for any $t\in U$ such that ${p_t+\dots+p_t\in\mathcal{P}\cap\mathcal{Y}}$, the line bundle ${\mathcal{O}_{C_t}\left((g-1)p_t\right)}$ is a theta-characteristic having ${h^0\left(C,(g-1)p_t\right)\geq r+1}$ and ${h^0\left(C,(g-1)p_t\right)\equiv r+1\,(\mathrm{mod}\,2)}$.
Thus ${[C_t,p_t]\in\mathcal{G}_{g}^r}$.

We note further that the map ${U\longrightarrow \mathcal{M}_{g,1}}$ given by ${t\longmapsto [C_t,p_t]}$ is finite.
Indeed it factors through the injection ${U/\Aut(C,L,p)\hookrightarrow \mathcal{S}_{g,1}}$ and the natural map $\mathcal{S}_{g,1}\longrightarrow \mathcal{M}_{g,1}$ defined by ${[C_t,L_t,p_t]\longmapsto [C_t,p_t]}$.
Hence each irreducible component $\mathcal{Z}\subset \mathcal{G}_{g}^r$ passing through $[C,p]$ satisfies
\begin{equation}\label{equation DIMENSION}
\dim \mathcal{Z} \geq \dim \left(\mathcal{P}\cap \mathcal{Y}\right) \geq \dim \mathcal{P} + \dim \mathcal{Y} - \dim \mathcal{C}^{(g-1)}.
\end{equation}
Finally, since ${\dim \mathcal{P}=\dim U=3g-2}$,  $\dim \mathcal{Y}\geq\dim U^r + r\geq 3g-2-\frac{r(r+1)}{2}+r$ by Theorem \ref{theorem COLOMBO}, and ${\dim \mathcal{C}^{(g-1)}=\dim U + g-1 = 4g-3}$, we deduce ${\dim \mathcal{Z} \geq  2g-\frac{r(r-1)}{2}-1}$ as claimed.
\end{proof}

In analogy with \cite[Proposition 2.4]{Fa}, the following result shows that if $\mathcal{G}_{g-1}^r$ possesses an irreducible component of the `expected dimension', its points can be suitably glued to elliptic curves, so that the resulting nodal curves in $\overline{\mathcal{M}}_{g,1}$ can be smoothed preserving a subcanonical point and the number of sections of the associated theta-characteristic.
\begin{theorem}\label{theorem INDUCTION}
Fix $r\geq 1$ and $g\geq 4$. If the locus $\mathcal{G}_{g-1}^r$ has an irreducible component $\mathcal{Z}_{g-1}$ of dimension $2g-\frac{r(r-1)}{2}-3$, then $\mathcal{G}_{g}^r$ possesses a non-empty irreducible component $\mathcal{Z}_{g}$ of dimension $2g-\frac{r(r-1)}{2}-1$.
\end{theorem}
\begin{proof}
We consider the locus
\begin{equation}\label{equation W}
\mathcal{W}:=\left\{ [C\cup_q E,p]\in \overline{\mathcal{M}}_{g,1}\left|
\begin{array}{l}
[C,q]\in \mathcal{Z}_{g-1},\, E \textrm{ is an elliptic curve}, \\
p\in E,\, p-q\in E\textrm{ is a torsion point of}\\
\textrm{order exactly }2g-2
\end{array}\right.\right\}.
\end{equation}
Our aim is to prove that $\overline{\mathcal{W}}$ is an irreducible component of the boundary of $\overline{\mathcal{G}}_{g}^r\subset \overline{\mathcal{M}}_{g,1}$.
Then we shall deduce the existence of an irreducible component $\mathcal{Z}_g\subset\mathcal{G}_{g}^r$ whose closure contains $\mathcal{W}$, so that $\mathcal{Z}_g$ shall have the expected dimension.

Using notation of Section \ref{subsection THETA-CHARACTERISTICS}, let $\mathcal{S}^*_{g,1}:=\mathcal{S}_{g,1}\cup A^+_{1,1}\cup A^-_{1,1}$ be the partial compactification of $\mathcal{S}_{g,1}$ such that the marked point on a nodal curve lies on an elliptic tail having even theta-characteristic.
Let $[C\cup_q E,p]\in \mathcal{W}$ be a general point and consider the class ${[X,N,\alpha_N,p]\in \mathcal{S}^*_{g,1}}$, where
$X=C\cup_{q'} R\cup_{q''}E$ is the blow-up of $C\cup_q E$ at $q$, and $N$ is the line bundle such that ${N_{|C}=\mathcal{O}_C\left((g-2)q'\right)}$, ${N_{|R}=\mathcal{O}_R\left(1\right)}$ and ${N_{|E}=\mathcal{O}_E\left((g-1)(p-q'')\right)}$.
Firstly, we want to show that the triple $[X,N,\alpha_N]$ lies in the boundary of the locus $\mathcal{S}^r_g$ of smooth spin curves possessing a theta characteristic with at least $r+1$ sections and the same parity as $r+1$. Namely,

\begin{claim}\label{claim SMOOTHING THETA}
${[X,N,\alpha_N]\in \overline{\mathcal{S}}^r_{g}}$ and $[C\cup_q E]\in \overline{\mathcal{M}}^r_g$.
\end{claim}
\begin{proof}[Proof of Claim \ref{claim SMOOTHING THETA}]
By Lemma \ref{lemma FARKAS SMOOTHING THETA}, we only need to show that ${[C,N_{|C}]\in \mathcal{S}_{g-1}}$ lies on an irreducible component of $\mathcal{S}^r_{g-1}$ having codimension ${\frac{r(r+1)}{2}}$.
Arguing as in the proof of Theorem \ref{theorem BOUND}, we achieve an inequality analogous to (\ref{equation DIMENSION}), that is
\begin{equation*}
\dim \mathcal{Z}_{g-1} \geq \dim \mathcal{P}' + \dim \mathcal{Y}' - \dim \mathcal{C}^{(g'-1)},
\end{equation*}
where $g'=g-1$ is the genus of $C$ and $\mathcal{P}',\mathcal{Y}',\mathcal{C}^{(g'-1)}$ are defined analogously.
In particular, they satisfy ${\dim \mathcal{P}'=\dim U=3g'-2}$, ${\dim \mathcal{Y}'\geq\dim U^r + r}$ and ${\dim \mathcal{C}^{(g'-1)}=\dim U+g'-1=4g'-3}$.
Since $\dim \mathcal{Z}_{g-1}= 2g'-\frac{r(r-1)}{2}-1$ by assumption, we deduce that $\dim U^r\leq 3g'-2 - \frac{r(r+1)}{2}$, and Theorem \ref{theorem COLOMBO} assures that the latter is actually an equality.
Thus any irreducible component of $\mathcal{S}^r_{g-1,1}$ containing $\mathcal{Z}_{g-1}$ has codimension ${\frac{r(r+1)}{2}}$ in $\mathcal{S}_{g-1,1}$.
Hence the image of such a component under the forgetful map ${\mathcal{S}_{g-1,1}\longrightarrow \mathcal{S}_{g-1}}$ is an irreducible component of $\mathcal{S}^r_{g-1}$ having codimension $\frac{r(r+1)}{2}$ and passing through $[C,N_{|C}]$.
Thus ${[X,N,\alpha_N]\in \overline{\mathcal{S}}^r_{g}}$ and $[C\cup_q E]\in \overline{\mathcal{M}}^r_g$ by Lemma \ref{lemma FARKAS SMOOTHING THETA}.
\end{proof}

Let ${\left(\mathcal{X}\stackrel{\phi}{\longrightarrow} B,  \xi, \alpha\colon \xi^2\longrightarrow \omega_\phi, B\stackrel{\rho}{\longrightarrow} \mathcal{X}\right)}$ be the restriction to $\mathcal{S}^*_{g,1}$ of a versal deformation family of ${(X,N,\alpha_N,p)}$ in $\overline{\mathcal{S}}_{g,1}$, with $X_b:=\phi^{-1}(b)$, $\xi_b:=\xi_{|X_b}$, $\alpha_b:=\alpha_{|\xi_b}$ and $p_b:=\rho(b)$.
Hence ${(X_{0},\xi_{0},\alpha_{0},p_{0})=(X,N,\alpha_N,p)}$ for some $0\in B$, and we have a commutative diagram of finite maps
\begin{equation*}
\xymatrix{ B \ar[r]^\mu \ar[d]_{\psi} & \overline{\mathcal{S}}_{g,1}\ar[d]^{\pi_1} \\  \widetilde{B} \ar[r]^\nu & \overline{\mathcal{M}}_{g,1}}
\end{equation*}
where $\widetilde{B}$ is a versal deformation space of $(C\cup_q E,p)$ in $\overline{\mathcal{M}}_{g,1}$.
Moreover, up to shrinking $B$, we can assume that $p_b\in X_b\smallsetminus\Sing(X_b)$ for any $b\in B$.

Let ${B^r:=\left\{b\in B\left|[X_b,\xi_b,\alpha_b]\in \overline{\mathcal{S}}_{g}^r\right.\right\}}$.
Thanks to Claim \ref{claim SMOOTHING THETA} and Theorem \ref{theorem COLOMBO}, we have that $0\in B^r$ and $\dim B^r\geq 3g-2-\frac{r(r+1)}{2}$.
Clearly, if $b\in B^r$ and $X_b$ is smooth, then $X_b$ coincides with its stable model $\widetilde{X}_b$ and the theta-characteristic $\xi_b$ gives a $\mathfrak{g}^r_{g-1}$ on $\widetilde{X}_b$, $L_b:=\left(\xi_b,V_{\widetilde{X}_b}\right)$, with $\xi_b^2=\omega_{\widetilde{X}_b}$ and $V_{\widetilde{X}_b}\subset H^0(X_b,\xi_b)$.
On the other hand, if ${X_b=C_b\cup_{q'_b} R_b\cup_{q''_b}E_b}$ for some $b\in B^r$, then Lemma \ref{lemma FARKAS LIMIT g^r_(g-1)} assures that the stable model $\widetilde{X}_b=C_b\cup_{q_b}E_b\in \overline{\mathcal{M}}^r_g$ admits a limit $\mathfrak{g}^r_{g-1}$, we denote by $L_b:=\left\{\left(\mathcal{L}_{C_b},V_{C_b}\right)\left(\mathcal{L}_{E_b},V_{E_b}\right)\right\}$, such that
\begin{equation}\label{equation LOCALLY CLOSED PROOF}
\mathcal{L}^{2}_{C_b}\cong \omega_{C_b}(2q_b) \quad \textrm{and} \quad \mathcal{L}^{2}_{E_b}\cong \mathcal{O}_{E_b}\left((2g-2)q_b\right).
\end{equation}
In particular, conditions (\ref{equation LOCALLY CLOSED PROOF}) are necessary to a limit $\mathfrak{g}^{g-1}_{2g-2}$ $\{(\mathcal{L}^{2}_{C_b},W_{C_b}),(\mathcal{L}^{2}_{E_b},W_{E_b})\}$ for being a canonical limit linear series on $\widetilde{X}_b$ (see \cite[p. 363]{EH1}), and they provide locally closed conditions in the space of limit $\mathfrak{g}^r_{g-1}$ on $\widetilde{X}_b$.
We note further that by Lemma \ref{lemma LIMIT g^r_(g-1)}, the curve $C\cup_q E$ admits a limit $\mathfrak{g}^r_{g-1}$
${\displaystyle L=\left\{\left(\mathcal{L}_C,V_C\right),\left(\mathcal{L}_E,V_E\right)\right\}}$ satisfying conditions (\ref{equation LOCALLY CLOSED PROOF}) and such that $a_r^L(p)= g-1$, $\mathcal{L}_C=\mathcal{O}_C((g-1)q)$ and $\mathcal{L}_E=\mathcal{O}_E((g-1)p)$.
Following the recipe of \cite[Section 3]{EH1}, we can thus base-change $B^r$ to get a family of triples $\left\{\left.(\widetilde{X}_b,p_b,L_b)\right|b\in B^r\right\}$, where $(\widetilde{X}_b, p_b)$ is the stable pointed curve associated to $(X_b,p_b)$, the limit linear series $L_b$ is a limit $\mathfrak{g}^r_{g-1}$ as above, and $(\widetilde{X}_{0},p_{0},L_{0})=(C\cup_q E,p,L)$.

Then we argue as in the proof of Theorem \ref{theorem BOUND} and we consider the $(g-1)$-fold relative symmetric product ${\widetilde{\mathcal{X}}^{(g-1)}\longrightarrow B^r}$ of the family $\widetilde{\mathcal{X}}\longrightarrow B^r$, so that the fibre over each $b$ is the $(g-1)$-fold symmetric product $\widetilde{X}_b^{(g-1)}$ of the stable curve $\widetilde{X}_b$.
We define two subvarieties of $\widetilde{\mathcal{X}}^{(g-1)}$ as ${\widetilde{\mathcal{P}}:=\left\{\left. p_b+\dots+p_b\in \widetilde{X}_b^{(g-1)}\right|b\in B^r\right\}}$, which is the image of the natural map $\widetilde{\rho}^{(g-1)}\colon B^r\longrightarrow \widetilde{\mathcal{X}}^{(g-1)}$ induced by $\rho$, and
\begin{equation*}
\widetilde{\mathcal{Y}}:=\left\{ p_1+\dots+p_{g-1}\in \widetilde{X}_b^{(g-1)}\left|
\begin{array}{l}
\,b\in B^r \textrm{ and } p_1+\dots+p_{g-1}=(\sigma) \\ \textrm{for some global section }\sigma \textrm{ of }L_b
\end{array}\right.\right\},
\end{equation*}
which is the locus of divisors $D_b\in \widetilde{X}_b^{(g-1)}$ where the sections of the limit linear series vanish.
Consider the locus $B':=\left\{b\in B^r \left|p_b+\dots+p_b\in \widetilde{\mathcal{P}}\cap \widetilde{\mathcal{Y}}\right.\right\}$, and set ${\mathcal{B}':=(\pi_1\circ \mu)(B')\subset \overline{\mathcal{M}}_{g,1}}$.
We note that by construction of the versal deformation family in $\mathcal{S}^*_{g,1}$, the locus $\mathcal{B}'$ lies in $\mathcal{M}_{g,1}\cup \Delta_{1,1}$, that is either $X_b$ is a smooth curve with $p_b\in X_b$, or $X_b=C_b\cup_{q_b}E_b$ is a nodal curve with $g(E_b)=1$ and $p_b\in E_b$.
If $\widetilde{X}_b$ is a smooth curve, then $b\in B'$ if and only if there exists a section $\sigma\in H^0(\widetilde{X}_b,\xi_b)$ such that $\ord_{p_b}(\sigma)= g-1$, i.e. $p_b$ is a subcanonical point of $\widetilde{X}_b$. Thus $[\widetilde{X}_b,p_b]\in \mathcal{G}^r_g$ and $\mathcal{B}'\cap \mathcal{M}_{g,1}\subset\mathcal{G}^r_{g}$.
On the other hand, looking at nodal curves ${\widetilde{X}_b=C_b\cup_{q_b}E_b}$ with a marked point $p_b\in E_b$ on the elliptic tail, we have the following.
\begin{claim}\label{claim LIMIT W}
The locus $\mathcal{W}\subset \Delta_{1,1}\subset\overline{\mathcal{M}}_{g,1}$ is an irreducible component of $\mathcal{B}'\cap \Delta_{1,1}$.
\end{claim}
\begin{proof}[Proof of Claim \ref{claim LIMIT W}]
Let $D\subset B'$ be an open disk centered at $0$.
Our aim is to show that---up to shrinking $D$---$(\pi_1\circ \mu)(D)\subset \mathcal{W}$, that is $[\widetilde{X}_b,p_b]\in \mathcal{W}$ for any $b\in D$.\\
We note that $[\widetilde{X}_b=C_b\cup_{q_b} E_b, p_b]\in \mathcal{B}'\cap \Delta_{1,1}$ if and only if there exists a global section $\sigma\in V_{E_b}\subset H^0(E_b, \mathcal{L}_{E_b})$ such that $\ord_{p_b}(\sigma)=g-1$, i.e. $a^{L_b}_r(p_b)=g-1$.
Thus Lemma \ref{lemma LIMIT g^r_(g-1)} guarantees that for any $b\in D$, the point $q_b\in C_b$ is a subcanonical point with $h^0(C_b,(g-2)q_b)\geq r$, and $p_b-q_b\in E_b$ is a torsion point of order $2g-2$.
In particular, the period of $p_b-q_b$ is exactly $2g-2$, because this is the period of $p-q$ on the special fibre $E$.\\
So we want to show that ${h^0\left(C_b,(g-2)q_b\right)\equiv r+1\,(\mathrm{mod}\,2)}$, which implies ${h^0(C_b,(g-2)q_b)\geq r+1}$ also.
We note that for any $b\in D$, we can endow the curve ${X_b=C_b\cup_{q'_b} R_b\cup_{q''_b}E_b}$ with a line bundle $N_b$ such that
$N_{b|C_b}:=\mathcal{O}_{C_b}\left((g-2)q'_b\right)$, $N_{b|E_b}:=\mathcal{O}_{E_b}\left((g-1)(p_b-q''_b)\right)$ and $N_{b|R_b}:=\mathcal{O}_{R_b}\left(1\right)$.
Hence it is naturally defined a family of pointed spin curves $(X_b, N_b, \beta_b, p_b)$ parameterized over $D$, whose central fibre is ${(X, N, \alpha_N, p)}$.
In particular ${X=C\cup_{q'} R\cup_{q''}E}$, with $[C,q]\in \mathcal{G}^r_{g-1}$, ${N_{|C}=\mathcal{O}_C\left((g-1)q'\right)}$ and ${N_{|E}=\mathcal{O}_E\left((g-1)(p-q'')\right)}$.
Hence ${h^0\left(C,N_{|C}\right)\equiv r+1\,(\mathrm{mod}\,2)}$, and $h^0(E,N_{|E})=0$ as the period of $p-q$ is exactly $2g-2$.
Since even and odd theta-characteristics do not mix (see Theorem \ref{theorem COLOMBO}), we deduce that ${h^0\left(C_b,N_{b|C_b}\right)\equiv r+1\,(\mathrm{mod}\, 2)}$ as well.
Thus $[C_b,q_b]\in \mathcal{G}^r_{g-1}$. \\
Finally, as $\mathcal{Z}_{g-1}\subset \mathcal{G}^r_{g-1}$ is an irreducible component and $[C,q]\in \mathcal{Z}_{g-1}$ is a general point, we conclude that $[C_b,q_b]\in \mathcal{Z}_{g-1}$.
Therefore $[\widetilde{X}_b=C_b\cup_{q_b} E_b,p_b]\in \mathcal{W}$, and hence $(\pi_1\circ \mu)(D)\subset \mathcal{W}$ as claimed.
\end{proof}

We point out that ${\dim \mathcal{B}'\geq 2g-\frac{r(r-1)}{2}-1}$.
Indeed
\begin{equation*}
\dim \mathcal{B}'=\dim B' \geq \dim \widetilde{\mathcal{P}} + \dim \widetilde{\mathcal{Y}} - \dim \widetilde{\mathcal{X}}^{(g'-1)},
\end{equation*}
where ${\dim \widetilde{\mathcal{P}}=\dim B^r\geq 3g-2-\frac{r(r+1)}{2}}$, ${\dim \widetilde{\mathcal{Y}}= \dim B^r+r}$ and ${\dim \widetilde{\mathcal{X}}^{(g-1)}=\dim B^r +g-1}$.
On the other hand, $\dim \mathcal{W}=\dim \mathcal{Z}_{g-1}+1= 2g-\frac{r(r-1)}{2}-2$.
Therefore $\dim \mathcal{B}'>\dim \mathcal{W}$, hence $\mathcal{B}'$ must parameterize also smooth curves by Claim \ref{claim LIMIT W}, that is $\mathcal{B}'\cap \mathcal{M}_{g,1}\subset\mathcal{G}^r_{g}$ is non-empty.
Thus there exists an irreducible component $\mathcal{Z}_g\subset \mathcal{G}_{g}^r$ such that its closure $\overline{\mathcal{Z}}_g$ contains $\mathcal{W}$.
In particular, $\overline{\mathcal{W}}$ is an irreducible component of $\overline{\mathcal{Z}}_g\cap \Delta_{1,1}$, and hence ${\dim \mathcal{Z}_g= \dim \mathcal{W}+1=2g-\frac{r(r-1)}{2}-1}$.
\end{proof}

\smallskip
We are now ready to prove Theorems \ref{theorem r=2} and \ref{theorem r=3}.
\begin{proof}[Proof of Theorem \ref{theorem r=2}]
Since $\mathcal{G}_{g}^2\subset \mathcal{G}_{g}^{\textrm{odd}}$ and the general point $[C,p]\in \mathcal{G}_{g}^{\textrm{odd}}$ is such that $h^0\left(C,(g-1)p\right)=1$ (see \cite[Theorem 2.1]{Bu}), we deduce $\dim \mathcal{G}_{g}^2< \mathcal{G}_{g}^{\textrm{odd}}=2g-1$.
Thus Theorem \ref{theorem BOUND} assures that any non-empty irreducible component of $\mathcal{Z}\subset \mathcal{G}_{g}^2$ has dimension $\dim \mathcal{Z}=2g-2$.
To show that $\mathcal{G}_{g}^2$ is non-empty for any $g\geq 6$ we argue by induction on $g$.
For $g=6$ the assertion follows from Example \ref{example G_6^2}, whereas Theorem \ref{theorem INDUCTION} leads to the assertion for any $g$.
\end{proof}

\begin{proof}[Proof of Theorem \ref{theorem r=3}]
By arguing as above, Example \ref{example G_9^3 minimal} gives that $\mathcal{G}^3_9$ possesses an irreducible component of dimension $2g-4$.
Then by induction on $g$, the assertion follows from Theorem \ref{theorem INDUCTION}.
\end{proof}

We recall that given a pointed curve $[C,p]\in \mathcal{G}^r_g$, the smallest possible canonical ramification sequence $\alpha^{K_C}(p)$ at $p$ is
\begin{equation}\label{equation MINIMAL RAMIFICATION 2}
\alpha:=(\underbrace{0,\dots,0}_{g-r-1},\underbrace{r,\dots,r}_{r},g-1),
\end{equation}
and the corresponding sequence of Weierstrass gaps is $(1,2,\dots,g-1-r,g,\dots,g+r,2g-1)$ (see Section \ref{subsection WEIERSTRASS POINTS}).
Then, using upper semi-continuity, Theorem \ref{theorem INDUCTION} gives a smoothing criterium for preserving the minimality of canonical ramification sequences, and hence the corresponding sequences of Weierstrass gaps.
Namely,
\begin{corollary}\label{corollary SMOOTHING WEIERSTRASS}
Under the assumption of Theorem \ref{theorem INDUCTION}, suppose that the general point $[C,q]\in \mathcal{Z}_{g-1}$ has canonical ramification sequence $\alpha^{K_C}(q)=(0,\dots,0,r,\dots,r,g-2)$, as in (\ref{equation MINIMAL RAMIFICATION 2}).
Then the general point of $\mathcal{Z}_{g}$ has canonical ramification sequence given by (\ref{equation MINIMAL RAMIFICATION 2}), as well.
\end{corollary}
\begin{proof}
Let $[C\cup_q E, p]\in \mathcal{W}$ as in (\ref{equation W}), so that $[C\cup_q E, p]\in \overline{\mathcal{Z}}_g\subset \overline{\mathcal{G}}_{g}^r$.
We then consider a family $\left(\mathcal{X}\stackrel{\phi}{\longrightarrow} T, T\stackrel{\rho}{\longrightarrow} \mathcal{X}\right)$ of pointed curves such that $[X_0,p_0]=[C\cup_q E, p]$ is the central fiber, and $[X_t,p_t]\in \mathcal{Z}_g$ is a smooth pointed curve for any $t\neq 0$.
The canonical bundles $\omega_{X_t}\cong \mathcal{O}_{X_t}\left((2g-2)p_t\right)$ induce a limit $\mathfrak{g}^{g-1}_{2g-2}$ on $C\cup_q E$,
$H:=\left\{\left(\mathcal{H}_{C},V_{C}\right),\left(\mathcal{H}_{E},V_{E}\right)\right\}$, where
\begin{equation*}
\mathcal{H}_{C}\cong \omega_C(2q)\cong \mathcal{O}_{C}\left((2g-2)q\right) \, \textrm{ and } \, \mathcal{H}_{E}\cong \omega_E\left((2g-2)q\right)\cong \mathcal{O}_{E}\left((2g-2)q\right)
\end{equation*}
(see \cite[p. 363]{EH1}). Furthermore, $H_C:=\left(\mathcal{H}_{C},V_{C}\right)$ is a complete linear series, that is $V_{C}=H^0\left(C, \mathcal{H}_{C}\right)$, because $\dim V_C=g$, the line bundle $\omega_C(q)\cong \mathcal{O}_{C}\left((2g-3)q\right)$ has a base point at $q$, and $\mathcal{H}_{C}$ is base-point free.

Letting $\alpha=(\alpha_0,\dots,\alpha_{g-1})$ be as in $(\ref{equation MINIMAL RAMIFICATION 2})$, we have that $\alpha^{H}(p)\geq\alpha^{\omega_{X_t}}(p_t)\geq \alpha$ by upper semi-continuity and minimality of $\alpha$ in $\mathcal{G}^r_g$ (see Remarks \ref{remark EARLY BOUND} and \ref{remark UPPER SEMI-CONTINUITY}).
In particular, setting $H_E:=\left(\mathcal{H}_{E},V_{E}\right)$, we have $\alpha^{H}(p)=\alpha^{H_E}(p)$ and the assertion shall follow from $\alpha^{H_E}(p)=\alpha$.

By the assumption on $\alpha^{K_C}(q)$ and the description of $H_C$, we have that its ramification sequence at $q\in C$ is $\alpha^{H_C}(q)=(0,1,\dots,1,r+1,\dots\,r+1,g-1)$, that is $\alpha^{H_C}(q)=(\alpha_0,\alpha_1+1,\dots,\alpha_{g-2}+1,\alpha_{g-1})$.
Since $H$ is a limit $\mathfrak{g}^{g-1}_{2g-2}$, compatibility conditions (\ref{equation COMPATIBILITY CONDITIONS}) give that the ramification sequence of $H_E$ at $q\in E$ is $\alpha^{H_E}(q)=(g-1-\alpha_{g-1},g-2-\alpha_{g-2},\dots,g-2-\alpha_{1},g-1-\alpha_{0})$.
Moreover, $\alpha_i^{H_E}(p)+\alpha_{g-1-i}^{H_E}(q)\leq g-1$ for any $0\leq i\leq g-1$, as $H_E$ is a $\mathfrak{g}^{g-1}_{2g-2}$ on $E$.
Thus $\alpha_0^{H_E}(p)=\alpha_0$, $\alpha_{g-1}^{H_E}(p)=\alpha_{g-1}$, and $\alpha_i\leq\alpha_i^{H_E}(p)\leq \alpha_i+1$ for any $1\leq i\leq g-2$.
In particular, $g-2\leq \alpha_i^{H_E}(p)+\alpha_{g-1-i}^{H_E}(q)\leq g-1$.
Then \cite[Proposition 5.2]{EH2} assures that for each $0\leq i\leq g-1$, if $\alpha_i^{H_E}(p)+ \alpha_{g-1-i}^{H_E}(q)= g-1$, then
\begin{equation*}
\left(\alpha_i^{H_E}(p)+i\right)p+ \left(\alpha_{g-1-i}^{H_E}(q)+g-1-i\right)q \sim (2g-2)p\,.
\end{equation*}
Since the period of the torsion point $p-q\in E$ is exactly $2g-2$, we conclude that $\alpha_i^{H_E}(p)+ \alpha_{g-1-i}^{H_E}(q)= g-1$ is satisfied only for $i=0$ and $i=g-1$.
Thus $\alpha_i^{H_E}(p)=\alpha_i$ for any $0\leq i\leq g-1$ and the assertion follows.
\end{proof}

Finally, using the latter corollary and arguing as in Theorems \ref{theorem r=2} and \ref{theorem r=3}, we deduce the following.
\begin{corollary}\label{corollary r=2 and 3}
For any $g\geq 6$, there exists an irreducible component of $\mathcal{G}^2_g$ whose general point $[C,p]$ has canonical ramification sequence  $\alpha^{K_C}(q)=(0,\dots,0,2,2,g-1)$.\\
Analogously, for any $g\geq 9$, there exists an irreducible component of $\mathcal{G}^3_g$ whose general point $[C,p]$ has canonical ramification sequence $\alpha^{K_C}(q)=(0,\dots,0,3,3,3,g-1)$.
\end{corollary}
\noindent In particular, the linear series $\left|(g-1)p\right|$ associated to such a general point is a base-point-free $\mathfrak{g}^r_{g-1}$ both when $r=2$ and $r=3$.

\section{Triply periodic minimal surfaces}\label{section PERIODIC MINIMAL SURFACES}

In this section we deal with the application to periodic minimal surfaces.
Initially, we recall the connection between periodic surfaces and theta-characteristics.
Then we follow \cite{Pi} and we state the main technical results leading to the proofs of Theorems \ref{theorem DENSITY} and \ref{theorem MINIMAL SURFACES}.

We assume hereafter that $X$ is a compact connected Riemann surface of genus $g\geq 3$.
In the following, we shall sometimes think $X$ as a complex algebraic curve, and we shall use the same notation of previous sections for the corresponding moduli spaces of Riemann surfaces.

We recall that a Riemann surface $X$ is \emph{periodic} if admits a conformal minimal immersion $\iota\colon X\longrightarrow \mathbb{T}^3$ into a flat 3-dimensional torus $\mathbb{T}^3=\mathbb{R}^3/\Lambda$.
Using the classical Generalized Weierstrass Representation (see e.g. \cite{La,Me,AP}), periodic surfaces may be characterized as follows (see \cite[Section 0]{Pi}).
Given a Riemann surface $X$, we denote by $G_\mathbb{Q}:=G\left(3, H^1\left(X,\mathbb{Q}\right)\right)$ and $G_\mathbb{R}:=G\left(3, H^1\left(X,\mathbb{R}\right)\right)$ the Grassmannians of 3-dimensional vector spaces in $H^1\left(X,\mathbb{Q}\right)$ and $H^1\left(X,\mathbb{R}\right)$, respectively, so that $G_\mathbb{Q}\subset G_\mathbb{R}$ is a dense inclusion.
Then $X$ is a periodic surface if and only if there exists a theta-characteristic $\mathcal{L}$ on $X$ such that $h^0(X,\mathcal{L})\geq 2$, the linear series $|\mathcal{L}|$ is base-point-free, and there exist two global sections $s_0,s_1\in H^0\left(X,\mathcal{L}\right)$ not vanishing simultaneously such that the holomorphic 1-forms $\omega_0,\omega_1,\omega_2\in H^0\left(X, \omega_X\right)$ defined as
\begin{equation*}
\omega_0:=s_0^2-s_1^2\quad \omega_1:=i(s_0^2+s_1^2)\quad \omega_2:=2s_0s_1
\end{equation*}
satisfy $\mathrm{Span}\left(\mathrm{Re}(\omega_0),\mathrm{Re}(\omega_1),\mathrm{Re}(\omega_2)\right)\in G_\mathbb{Q}$.
In this setting, the 4-tuple $(X,\mathcal{L},s_0,s_1)$ is the spinor representation of the periodic surface $X$ (see e.g. \cite{KS}).
Moreover, the set $\Lambda:=\left\{\left. \mathrm{Re}\left(\int_\gamma\omega_0,\int_\gamma\omega_1,\int_\gamma\omega_2\right) \in \mathbb{R}^3\right|\gamma\in H_1(X, \mathbb{Z})\right\}$ is a lattice of $\mathbb{R}^3$ of rank 3 and, after a translation, the conformal minimal immersion $\iota\colon X\longrightarrow \mathbb{R}^3/\Lambda$ is given by $\iota(p)= \mathrm{Re}\left(\small\int^p_{q}\omega_1,\int^p_{q}\omega_2,\int^p_{q}\omega_3\right)$, for some fixed point $q\in X$.

\smallskip
Thanks to the construction of \cite[Section 2]{Pi} and using notation of Section \ref{subsection THETA-CHARACTERISTICS}, we then introduce the moduli space $\mathcal{P}_g$ parameterizing 4-tuples $[X,\mathcal{L},s_0,s_1]$, where $[X,\mathcal{L}]\in \mathcal{S}_g^1\cup \mathcal{S}_g^2$ is a smooth spin Riemann surface such that $h^0\left(X,\mathcal{L}\right)\geq 2$, and $s_0,s_1\in H^0\left(X,\mathcal{L}\right)$ are two independent sections.

Given $[X,\mathcal{L},s_0,s_1]\in\mathcal{P}_g$, it is possible to consider its versal deformation space $\left(\mathcal{X}\stackrel{\phi}{\longrightarrow} B, \mathcal{N}, \sigma_0,\sigma_1\right)$ in $\mathcal{P}_g$ (see \cite[Section 3]{Pi}), where $\mathcal{X}\stackrel{\phi}{\longrightarrow} B$ is a family of Riemann surfaces $X_b:=\phi^{-1}(b)$ of genus $g$, and $\mathcal{N}$ is a line bundle on $\mathcal{X}$ with sections $\sigma_0,\sigma_1\in H^0\left(\mathcal{X},\mathcal{N}\right)$ restricting on each fibre to a point $[X_b, \mathcal{N}_b, \sigma_{0,b},\sigma_{1,b}]\in \mathcal{P}_g$.
In particular, both the natural modular maps $\mu\colon B\longrightarrow \mathcal{P}_g$ and $\nu\colon B\longrightarrow \mathcal{M}_g$ are surjective with discrete fibres.
We assume further that $B$ is simply connected and the family is complete, that is $\mu(B)$ contains an irreducible component of $\mathcal{P}_g$.
Moreover, we denote by $B(\mathbb{Q})\subset B$ sublocus of periodic surfaces.

\smallskip
The main technical result in \cite{Pi} connects variations of a minimal immersion preserving the real periods of $\displaystyle\left(\mathrm{Re}(\omega_0),\mathrm{Re}(\omega_1),\mathrm{Re}(\omega_2)\right)$ up to the first order, and periodic Jacobi fields along such a minimal immersion.
In particular, the variation of the period map at a point $[X,\mathcal{L},s_0,s_1]\in\mathcal{P}_g$ is related by \cite[Proposition 6.11]{Pi} to the kernel of the Schr\"{o}dinger operator $\Delta + \left|\nabla f\right|^2$ acting on real valued functions $u\in C^{\infty}(X)$, where $f\colon X\longrightarrow \mathbb{P}^1\cong S^2\subset \mathbb{R}^3$ is the holomorphic map defined by the linear system $|\Span(s_0,s_1)|$, $\Delta$ is the Laplace-Beltrami operator and $\nabla$ is the gradient (see \cite{MR}).
Then \cite[Theorem 2]{Na} and the Inverse Function Theorem lead to the following density result (see \cite[Proposition 6.12 and Theorem 6.14]{Pi}).
\begin{proposition}\label{proposition DENSITY}
Let $[X,\mathcal{L},s_0,s_1]\in \mathcal{P}_g$ be such that $s_0$ and $s_1$ do not have common zeroes, and $s_1$ vanishes at a single point.
Let $B$ be a versal deformation space of $(X,\mathcal{L},s_0,s_1)$ in $\mathcal{P}_g$.
If $\dim_{\mathbb{R}} B=6g$, then the locus $B(\mathbb{Q})$ is dense in $B$.
\end{proposition}

We can now prove Theorems \ref{theorem DENSITY} and \ref{theorem MINIMAL SURFACES}.

\begin{proof}[Proof of Theorem \ref{theorem DENSITY}]
The proof of assertions (i) and (iii) is actually included in \cite[Section 6]{Pi}.
Then we assume $g\geq 6$ and we follow the very same argument to prove assertion (ii), which states that the locus $\mathcal{R}_g^{\textrm{odd}}\subset \mathcal{M}_g$ parameterizing odd periodic surfaces is dense in the locus $\mathcal{M}_g^2$ of surfaces admitting an odd theta-characteristic with at least three global sections.
By Corollary \ref{corollary r=2 and 3}, the locus $\mathcal{G}^2_g\subset \mathcal{M}_{g,1}$ is non-empty, and its general point $[X,p]$ is such that $h^0(X,(g-1)p)=3$ and the linear series $|(g-1)p|$ is base-point-free.
Setting $\mathcal{L}:=\mathcal{O}_X\left((g-1)p\right)$, there exist two sections ${s_0,s_1\in H^0\left(X,\mathcal{L}\right)}$ without common zeroes, and we can assume that $s_1$ vanishes only at $p$ by definition of subcanonical point.
Then we consider a complete versal simply connected deformation space $B$ of $(X,\mathcal{L}, s_0,s_1)$ in $\mathcal{P}_g$.
We recall that any irreducible component of the locus of spin curves $\mathcal{S}^2_g$ has (complex) dimension $3g-6$ (see e.g. \cite{Te,Fa2}).
Moreover, any such a component passing through $[X,\mathcal{L}]$ has general point $[C,\xi]$ satisfying $h^0(C,\xi)=3$.
By completeness, $\mu(B)$ contains an irreducible component of $\mathcal{P}_g$.
Hence the real dimension of $B$ is $\dim_{\mathbb{R}}(B)=2\left(3g-6+2h^0(C,\xi)\right)=6g$.
Thus Proposition \ref{proposition DENSITY} assures that $B(\mathbb{Q})$ is dense in $B$.
By taking images under the modular map $\nu\colon B\longrightarrow \mathcal{M}_g$, we then deduce that $\mathcal{R}_g^{\textrm{odd}}$ is dense in $\mathcal{M}_g^{2}$.
\end{proof}

\begin{proof}[Proof of Theorem \ref{theorem MINIMAL SURFACES}]
The assertions on even and hyperelliptic minimal surfaces are included in \cite[Theorem 2]{Pi}.
Moreover, the proof of Theorem \ref{theorem DENSITY} assures that for any $g\geq 6$, there exists a point $[X,\mathcal{L},s_0,s_1]\in \mathcal{P}_g$ satisfying $h^0(X,\mathcal{L})=3$ and the hypothesis of Proposition \ref{proposition DENSITY}.
Thus the assertion of Theorem \ref{theorem MINIMAL SURFACES} for odd minimal surfaces follows from \cite[Section 6.19]{Pi}.
\end{proof}

\section*{Acknowledgements}
We are grateful to Gavril Farkas and Gabriele Mondello for helpful discussions.

\end{document}